\documentclass[12pt,leqno]{amsart}
\usepackage{latexsym,amsmath,amssymb,mathrsfs}

\title[Geodesics in the Heisenberg group]{Geodesics in the Heisenberg group}
\author{Piotr Haj\l{}asz and Scott Zimmerman}

\address{P.\ Haj{\l}asz: Department of Mathematics, University of Pittsburgh, 301
  Thackeray Hall, Pittsburgh, PA 15260, USA, {\tt hajlasz@pitt.edu}}

\address{S.\ Zimmerman: Department of Mathematics, University of Pittsburgh, 301
  Thackeray Hall, Pittsburgh, PA 15260, USA, {\tt srz5@pitt.edu}}

\thanks{P.H.\ was supported by NSF grant DMS-1161425.}

plus 6pt minus 12pt
\def\DD{\mathfrak{D}}

\newtheorem{theorem}{Theorem}

\newtheorem{corollary}[theorem]{Corollary}
\newtheorem{proposition}[theorem]{Proposition}

\theoremstyle{definition}
\newtheorem{remark}[theorem]{Remark}

\newcommand{\barint}{
\rule[.036in]{.12in}{.009in}\kern-.16in \displaystyle\int }

\newcommand{\barcal}{\mbox{$ \rule[.036in]{.11in}{.007in}\kern-.128in\int $}}

\newcommand{\bbbz}{\mathbb Z}
\newcommand{\bbbn}{\mathbb N}
\newcommand{\bbbr}{\mathbb R}

\newcommand{\bbbh}{\mathbb H}
\newcommand{\bbbc}{\mathbb C}

\newcommand{\Heis}{\mathbb H}

\def\mvint_#1{\mathchoice
          {\mathop{\vrule width 6pt height 3 pt depth -2.5pt
                  \kern -8pt \intop}\nolimits_{\kern -3pt #1}}%
          {\mathop{\vrule width 5pt height 3 pt depth -2.6pt
                  \kern -6pt \intop}\nolimits_{#1}}%
          {\mathop{\vrule width 5pt height 3 pt depth -2.6pt
                  \kern -6pt \intop}\nolimits_{#1}}%
          {\mathop{\vrule width 5pt height 3 pt depth -2.6pt
                  \kern -6pt \intop}\nolimits_{#1}}}

\numberwithin{theorem}{section} \numberwithin{equation}{section}

\begin{document}

\subjclass[2010]{Primary 53C17; Secondary 42A05, 53C22}
\keywords{Heisenberg group, geodesics, Fourier series, isoperimetric inequality}
\sloppy


\sloppy

\begin{abstract}
We provide a new and elementary proof for the structure of geodesics in the Heisenberg group $\Heis^n$. The proof is based on a new isoperimetric inequality for
closed curves in $\bbbr^{2n}$. We also prove that the Carnot-Carath\'eodory metric is
real analytic away from the center of the group.
\end{abstract}

\maketitle

\section{Introduction}

The aim of this paper is to provide a detailed and a self-contained presentation of the structure of geodesics in the Heisenberg groups. While the paper is mostly of expository 
character, some results are new. The new results are: the isoperimetric inequality for closed curves in $\bbbr^{2n}$ (Theorem~\ref{CorIso}), 
a new proof of the structure of the geodesics in the Heisenberg group (Theorem~\ref{main}), and
the real analyticity of the Carnot-Carath\'eodory metric away from the center of the group (Theorem~\ref{main2}). 
The proof of the isoperimetric inequality (Theorem~\ref{CorIso}) is based on an elementary adaptation of the classical proof
of the isoperimetric inequality due to Hurwitz. This new isoperimetric inequality is then used to establish the structure of geodesics in $\bbbh^n$.
The proof of the real analyticity of the Carnot-Carath\'eodory metric follows from the explicit formulas for the geodesics in $\bbbh^n$ established in Theorem~\ref{main}.
We believe that, because of the increasing popularity of the geometric theory of the Heisenberg groups among young researchers and graduate students, such a
survey paper will be useful.

The Heisenberg group $\Heis^n$ is $\bbbc^n\times\bbbr=\bbbr^{2n+1}$ given the structure of a Lie group with multiplication
\begin{eqnarray*}
(z,t)*(z',t')
& = &
\Big(z+z',t+t'+2 {\rm Im}\, \sum_{j=1}^nz_j\overline{z_j'}\Big)\\
& = &
\Big(x_1+x_1', \dots , x_n + x_n' , y_1+y_1', \dots, y_n+y_n', t+t'+2 \sum_{j=1}^n (x_j' y_j - x_j y_j')\Big)
\end{eqnarray*}
with Lie algebra $\mathfrak{g}$ whose basis of left invariant vector fields at any $p=(x_1, \dots, x_n , y_1,\dots,y_n,t) \in \Heis^n$ is
$$
X_j(p) = \frac{\partial}{\partial x_j} + 2y_j \frac{\partial}{\partial t}, 
\quad 
Y_j(p)= \frac{\partial}{\partial y_j} - 2x_j \frac{\partial}{\partial t}, 
\quad 
T=\frac{\partial}{\partial t},
\quad
j=1,2,\ldots,n.
$$
We call $H \Heis^n = \mathrm{span} \{ X_1,\dots, X_n, Y_1,\dots,Y_n \}$ the {\em horizontal distribution} on $\Heis^n$, 
and denote by $H_p \Heis^n$ the horizontal space at $p$.
An absolutely continuous curve $\Gamma:[0,S] \to \bbbr^{2n+1}$ is said to be {\em horizontal} 
if $\dot{\Gamma}(s) \in H_{\Gamma(s)} \Heis^n$ for almost every $s \in [0,S]$.
It is easy to see that the horizontal distribution is the kernel of the {\em standard contact form}
$$
\alpha = dt+2 \sum_{j=1}^n(x_j dy_j - y_j dx_j).
$$
That is, $H_p \Heis^n = \text{ker} \, \alpha (p)$. 
Hence it follows that an absolutely continuous curve $\Gamma = (x_1,\dots, x_n, y_1,\dots,y_n,t)$ is horizontal if and only if 
$\dot{t}(s)=2 \sum_{j=1}^n (\dot{x}_j(s) y_j(s) - x_j(s) \dot{y}_j(s))$ for almost every $s \in [0,S]$.
This means that
\begin{equation}
\label{area1}
t(s)-t(0) = 2 \sum_{j=1}^n \int_0^s (\dot{x}_j(\tau) y_j(\tau) - x_j(\tau) \dot{y}_j(\tau)) \, d \tau 
\end{equation}
for every $s \in [0,S]$.
Suppose $\gamma = (x_1,\dots, x_n, y_1,\dots,y_n):[0,S] \to \bbbr^{2n}$ is absolutely continuous. 
If a value for $t(0)$ is fixed, then \eqref{area1} gives a unique horizontal curve $\Gamma=(\gamma(s),t(s))$ whose projection onto the 
first $2n$ coordinates equals $\gamma$. We call this curve $\Gamma$ a {\em horizontal lift} of $\gamma$.

If the curve $\gamma:[0,S] \to \bbbr^{2n}$ is closed and $\Gamma = (\gamma,t)$ is a horizontal lift of $\gamma$, then it follows from Green's Theorem 
that the total change in height $t(S)-t(0)$ equals $-4$ times the sum of the signed areas enclosed by the projections $\gamma_j$ of the curve to each $x_jy_j$-plane.
Since the curves in the $x_jy_j$-planes may have self-intersections, the signed area has to take multiplicity of the components of the complement of $\gamma_j$
into account. This multiplicity can be defined in terms of the winding number. We will not provide more details here as this interpretation of the
change of the height will not play any role in our argument, and it is mentioned here only to give an additional vantage point to  
the geometry of the problem.

We equip the horizontal distribution $H \Heis^n$ with the left invariant Riemannian metric so that the vector fields $X_j$ and $Y_j$ are orthonormal at every point $p\in \Heis^n$. 
Note that this metric is only defined on $H\Heis^n$ and not on $T\Heis^n$. 
If
$$
v=\sum_{j=1}^n \Big( a_j \frac{\partial}{\partial x_j}\Big|_p 
                   + b_j\frac{\partial}{\partial y_j}\Big|_p\Big) +c\frac{\partial}{\partial t}\Big|_p\in H_p\Heis^n,
$$
then it is easy to see that 
$$
v=\sum_{j=1}^n a_j X_j(p)+b_j Y_j(p)
$$
and hence $c=2\sum_{j=1}^n(a_j y_j(p)-b_jx_j(p))$.
Clearly $|v|_H\leq |v|_E$, where
$$
|v|_H=\sqrt{\sum_{j=1}^n a_j^2+b_j^2}
\quad
\text{and}
\quad
|v|_E= \sqrt{\Big(\sum_{j=1}^n a_j^2+b_j^2\Big)+c^2}
$$
are the lengths of $v$ with respect to the metric in $H_p\Heis^n$ and the
Euclidean metric in $\bbbr^{2n+1}$ respectively. 
On compact sets in $\bbbr^{2n+1}$ the coefficient $c$ is bounded by $C|v|_H$, so
for any compact set $K\subset\bbbr^{2n+1}$
there is a constant $C(K)\geq 1$ such that
\begin{equation}
\label{loclip}
|v|_H\leq |v|_E\leq C(K)|v|_H
\quad
\text{for all $p\in K$ and $v\in H_p\Heis^n$.}
\end{equation}
A horizontal curve $\Gamma:[a,b]\to\Heis^n$ satisfies
$$
\dot{\Gamma}(s)=\sum_{j=1}^n \dot{x}_j(s)X_j(\Gamma(s))+\dot{y}_j(s)Y_j(\Gamma(s)).
$$
Its length with respect to the metric in $H\Heis^n$ equals
\begin{equation}
\label{h1}
\ell_H(\Gamma) := \int_a^b |\dot{\Gamma}(s)|_H\, ds=\int_a^b \sqrt{\sum_{j=1}^n \left( \dot{x}_j(s)^2+\dot{y}_j(s)^2 \right)} \, ds.
\end{equation}
Notice that the length $\ell_H(\Gamma)$ of a horizontal curve $\Gamma=(\gamma,t)$ given by \eqref{h1} equals the usual Euclidean length 
$\ell_E(\gamma)$ in $\bbbr^{2n}$ of the projection $\gamma$.

The {\em Carnot-Carath\'eodory} metric $d_{cc}$ in $\Heis^n$ is defined as the infimum of lengths $\ell_H(\Gamma)$
of horizontal curves connecting two given points. It is well known that any two points can be connected by a horizontal curve (we will actually prove it), and hence $d_{cc}$
is a true metric. It follows from \eqref{loclip} that for any compact set $K\subset\bbbr^{2n+1}$ there is a constant $C=C(K)$ such that
\begin{equation}
\label{EH}
|p-q|\leq C d_{cc}(p,q)
\qquad
\text{for all $p,q\in K$.}
\end{equation}
If $\Gamma:[a,b]\to X$ is a (continuous) curve in any metric space $(X,d)$, then the {\em length} of $\Gamma$ is defined by 
$$
\ell_d(\Gamma)=\sup\sum_{i=0}^{n-1} d(\Gamma(s_i),\Gamma(s_{i+1})),
$$
where the supremum is taken over all $n\in\bbbn$ and all partitions $a=s_0\leq s_1\leq\ldots\leq s_n=b$.
In particular, if $\Gamma:[a,b]\to\Heis^n$ is a curve in the Heisenberg group, then its length with respect to the Carnot-Carath\'eodory metric is
\begin{equation}
\label{h2}
\ell_{cc}(\Gamma)=\sup\sum_{i=0}^{n-1} d_{cc}(\Gamma(s_i),\Gamma(s_{i+1})).
\end{equation}
It follows immediately from the definition that if $\Gamma$ is a horizontal curve, then 
$\ell_{cc}(\Gamma)\leq \ell_H(\Gamma)$; hence every horizontal curve is {\em rectifiable}
(i.e. of finite length), 
but it is not obvious, whether $\ell_{cc}(\Gamma)=\ell_H(\Gamma)$.
It is also not clear whether every rectifiable curve
can be reparametrized as a horizontal curve. Actually all of this is true and we have
\begin{proposition}
\label{prop}
In $\Heis^n$ we have
\begin{enumerate}
\item Any horizontal curve $\Gamma$ is rectifiable and $\ell_{cc}(\Gamma)=\ell_H(\Gamma)$.
\item Lipschitz curves are horizontal.
\item Every rectifiable curve admits a $1$-Lipschitz parametrization (and hence it is horizontal with respect to this 
parametrization).
\end{enumerate}
\end{proposition}
Thus, up to a reparametrization, the class of rectifiable curves coincides with the class of horizontal ones.
Proposition~\ref{prop} is well known, but it is difficult to find a good reference that would provide a straightforward proof. 
For the sake of completeness, we decided to provide a proof of this result in the Appendix.
For a proof of (1) in more generality, see Theorem 1.3.5 in \cite{montithesis}.

A {\em geodesic} from $p$ to $q$ is a curve of shortest length between the two points. 
It is already well known that any two points in the Heisenberg group can be connected by a geodesic. Also, the structure of every geodesic in a form of an 
explicit parameterization is known. The proofs in the case of $\Heis^1$ can be found in \cite{belaiche,book,gaveau,montgomery}, and the general case of $\Heis^n$
is treated in \cite{ambrosiorigot,berest,monti}.

If $n=1$, the structure of geodesics can be obtained via the two dimensional isoperimetric inequality (see \cite{belaiche,book,montgomery}).
Consider a horizontal curve $\Gamma=(\gamma,t)$ in $\Heis^1$ connecting the origin to some point $q=(0,0,T)$ with $T \neq 0$.
The length of $\Gamma$ equals the length of its projection $\gamma$ to $\bbbr^2$ (which is a closed curve).
Also, by \eqref{area1}, the change $T$ in the height of $\Gamma$ must equal $-4$ times the signed area enclosed by $\gamma$.
Thus the projection of any horizontal curve connecting $0$ to $q$ must enclose the same area $|T|/4$, and finding a geodesic which connects $0$ to $q$ 
reduces to a problem of finding a shortest closed curve $\gamma$ enclosing a fixed area.
Thus the classical isoperimetric inequality implies that $\Gamma$ will have smallest length when $\gamma$ is a circle.
Then the $t$ component of $\Gamma$ is determined by \eqref{area1} and one obtains an explicit parametrization of the geodesics in $\Heis^1$ 
connecting the origin to a point on the $t$-axis.
Such geodesics pass through all points $(x_0,y_0,t_0)$, $t_0\neq 0$ in $\Heis^1$. If $q=(x_0,y_0,0)$, then it is easy to see that
the segment $\overline{0q}$ connecting the origin to $q$ is a geodesic. This describes all geodesics connecting the origin to 
any other point in $\Heis^1$.
Due to the left-invariance of the vector fields $X$ and $Y$, parameterizations for geodesics between arbitrary points in $\Heis^1$ may be 
found by left multiplication of the geodesics discussed above.

This elegant argument, however, does not apply to $\Heis^n$ when $n>1$ and known proofs of the structure of geodesics in $\Heis^n$ are based on the Pontryagin maximum principle
\cite{ambrosiorigot,berest,monti}. 
In this paper we will provide a straightforward and elementary argument leading to an explicit parameterization of geodesics in $\Heis^n$ (Theorem~\ref{main}). 
Our argument is 
based on Hurwitz's proof \cite{hurwitz}, of the isoperimetric inequality in $\bbbr^2$ involving Fourier series.
The Hurwitz argument is used to prove a version of the isoperimetric inequality for closed curves in $\bbbr^{2n}$ (Theorem~\ref{CorIso}).
This isoperimetric inequality allows us to extend the isoperimetric proof of the structure of geodesics in $\bbbh^1$ to the higher dimensional case $\bbbh^n$ as seen in the proof of Theorem~\ref{main}.
For a related, but different isoperimetric inequality in $\bbbr^{2n}$, see \cite{schoen}. 

As an application of our method we also prove that the
Carnot-Carath\'eodory metric is real analytic away from the center of the group (Theorem~\ref{main2}). This improves a result of Monti \cite{montithesis,monti} who proved
that this distance is $C^\infty$ smooth away from the center. We also find a formula for the Carnot-Carath\'eodory distance (Corollary~\ref{2015}) that, we hope, 
will find application in the study of geometric properties of the Heisenberg groups. 

The paper is organized as follows.
In Section~\ref{proof} we will state and prove the isoperimetric inequality and
the result about the structure of the geodesics in $\Heis^n$, 
and in Section~\ref{analytic} we use this structure to show that the distance function in $\bbbh^n$ is analytic away from the center.
In Section~\ref{rotation} we address the problem of comparing different geodesics (there are infinitely many of them) connecting
the origin to a point $(0,0,T)\in\Heis^n$ on the $t$-axis.
Finally, in the Appendix we prove Proposition~\ref{prop}.

\noindent
{\bf Acknowledgements.}
The authors would like to thank the referee for valuable comments that led to an improvement of the paper.

\section{The isoperimetric inequality and the structure of geodesics} 
\label{proof}

Any horizontal curve $\Gamma$ is rectifiable, and we may use the arc-length parametrization to assume that the speed $|\dot{\Gamma}|_H$ of
$\Gamma:[0,\ell_H(\Gamma)]\to\Heis^n$ equals $1$. Then, we can reparametrize it as a curve of constant speed defined on $[0,1]$, and hence we can assume
that $\Gamma:[0,1]\to\Heis^n$ satisfies
\begin{equation}
\label{speed}
\sum_{j=1}^n \dot{x}_j(s)^2+\dot{y}_j(s)^2 = \ell_H(\Gamma)^2=\ell_{cc}(\Gamma)^2
\quad
\mbox{for almost all $s\in [0,1]$.}
\end{equation}
On the other hand any rectifiable curve in $\Heis^n$ can be reparametrized as a horizontal curve via the arc length parameterization (Proposition~\ref{prop}),
and thus, when looking for length minimizing curves (geodesics), it suffices to restrict our attention to horizontal curves
$\Gamma:[0,1]\to\Heis^n$ satisfying \eqref{speed}.

Since the left translation in $\Heis^n$ is an isometry, it suffices to investigate geodesics connecting the origin $0\in\Heis^n$
to another point in $\Heis^n$. Indeed, if $\Gamma$ is a geodesic connecting $0$ to $p^{-1}*q$, then $p*\Gamma$ is a geodesic
connecting $p$ to $q$.

If $q$ belongs to the subspace $\bbbr^{2n}\times\{ 0\}\subset\bbbr^{2n+1}=\Heis^n$, then it is easy to check that the straight line $\Gamma(s)=sq$, $s\in [0,1]$
is a unique geodesic (up to a reparametrization) connecting $0$ to $q$. Indeed, it is easy to check that $\Gamma$ is horizontal, and its length 
$\ell_{cc}(\Gamma)=\ell_H(\Gamma)$ equals the 
Euclidean length $|\overline{0q}|$ of the segment $\overline{0q}$ because $\Gamma$ is equal to its projection $\gamma$. For any other horizontal
curve $\tilde{\Gamma}=(\tilde{\gamma},\tilde{t})$ connecting $0$ to $q$, the projection $\tilde{\gamma}$ on $\bbbr^{2n}$ would not be a segment 
(since horizontal lifts of curves are unique up to vertical shifts), and hence we would have
$\ell_{cc}(\tilde{\Gamma})=\ell_H(\tilde{\Gamma})=\ell_E(\tilde{\gamma})>|\overline{0q}|=\ell_{cc}(\Gamma)$
which proves that $\tilde{\Gamma}$ cannot be a geodesic.

In Theorem~\ref{main}, we will describe the structure of geodesics in $\Heis^n$ connecting the origin to a point $(0,0,T)\in\bbbr^{2n}\times\bbbr=\Heis^n$,
$T\neq 0$, lying on the $t$-axis. Later we will see (Corollary~\ref{cor}) that these curves describe all geodesics in $\Heis^n$ connecting $0$ to
$q\not\in\bbbr^{2n}\times \{ 0\}$. The geodesics connecting $0$ to $q\in\bbbr^{2n}\times \{ 0\}$ have been described above.

\begin{theorem}
\label{main}
A horizontal curve 
$$
\Gamma(s)=(x(s),y(s),t(s))=(x_1(s),\dots, x_n(s), y_1(s),\ldots,y_n(s),t(s)):[0,1]\to\Heis^n
$$ 
of constant speed, connecting the origin $\Gamma(0)=(0,0,0)\in\bbbr^{2n}\times\bbbr=\Heis^n$ to a point 
$\Gamma(1)=(0,0,\pm T)$, $T> 0$, on the $t$-axis is a geodesic if and only if
\begin{equation} 
\label{param}
\begin{aligned}
x_j(s) = A_j (1- \cos(2 \pi s)) \mp B_j \sin(2 \pi s) \\
y_j(s) = B_j (1- \cos(2 \pi s)) \pm A_j \sin(2 \pi s)
\end{aligned}
\end{equation}
for $j=1,2,\ldots,n$ and
$$
t(s) = \pm T \left(s - \frac{\sin (2 \pi s)}{2 \pi} \right)
$$
where $A_1,\dots, A_n,B_1,\ldots,B_n$ are any real numbers such that $4\pi\sum_{j=1}^n (A_j^2+B_j^2)=T$.
\end{theorem}
\begin{remark}
Observe that if $\Gamma(1)=(0,0,+T)$, the equations \eqref{param} give a constant-speed
parametrizations of negatively oriented circles
in each of the $x_jy_j$-planes, centered at $(A_j,B_j)$, and of radius
$\sqrt{A_j^2+B_j^2}$. Each circle passes through the origin at $s=0$. The signed area of such a circle equals $-\pi(A_j^2+B_j^2)$. Thus the change in height $t(1)-t(0)$ 
which is $-4$ times the sum of the signed areas of the projections of the curve on the $x_jy_j$-planes equals
$$
(-4)\sum_{j=1}^n \big(-\pi(A_j^2+B_j^2)\big)= T.
$$
Clearly this must be the case, because $\Gamma$ connects the origin to $(0,0,T)$.
Any collection of circles in the $x_jy_j$-planes passing through the origin and having radii $r_j\geq 0$ are projections of 
a geodesic connecting the origin to the point $(0,0,T)$ where $T=4\pi\sum_{j=1}^nr_j^2$. In particular we can find a geodesic for which only one projection
is a nontrivial circle (all other radii are zero) and another geodesic for which all projections are non-trivial circles. That suggests that the geodesics 
connecting $(0,0,0)$ to $(0,0,T)$ may have many different shapes. This is, however, an incorrect intuition. As we will see in Section~\ref{rotation},
all such geodesics are obtained from one through a rotation of $\bbbr^{2n+1}$ about the $t$-axis. This rotation is also an isometric mapping of $\Heis^n$. 
The above reasoning applies also to the case when $\Gamma(1)=(0,0,-T)$ with the only difference being that the circles are positively oriented.
\end{remark}
\begin{remark}
The parametric equations for the geodesics can be nicely expressed with the help of complex numbers, see \eqref{complex}.
\end{remark}

The proof is based on the following version of the isoperimetric inequality which is of independent interest.

In the theorem below we use identification of  $\bbbr^{2n}$ with $\bbbc^n$ given by
$$
\bbbr^{2n}\ni (x,y)=(x_1,\ldots,x_n,y_1,\ldots,y_n) \leftrightarrow (x_1+iy_1,\ldots,x_n+iy_n)=x+iy\in\bbbc^n.
$$

Every rectifiable curve $\gamma$ admits the arc-length parametrization. By rescaling it, we may assume that $\gamma$
is a constant speed curve defined on $[0,1]$.
\begin{theorem}
\label{CorIso}
If $\gamma=(x_1,\ldots,x_n,y_1,\ldots,y_n):[0,1]\to\bbbr^{2n}$ is a closed rectifiable curve prametrized by constant speed, then
\begin{equation}
\label{e2015}
L^2\geq 4\pi|\DD|,
\end{equation}
where
$L$ is the length of $\gamma$ and $\DD=\DD_1+\ldots+\DD_n$ is the sum of signed areas enclosed by the curves
$\gamma_j=(x_j,y_j):[0,1]\to\bbbr^2$, i.e.
$$
\DD_j=\frac{1}{2}\int_0^1(\dot{y}_j(s)x_j(s)-\dot{x}_j(s)y_j(s))\, ds.
$$
Moreover, equality in \eqref{e2015} holds if and only if 
there are points $A,B,C,D\in\bbbr^n$ such that $\gamma$ has the form
\begin{equation}
\label{Eq1}
\gamma(s)=(C+iD)+(1-e^{+2\pi is})(A+iB),
\quad
\text{when}
\quad
L^2=4\pi\DD
\end{equation}
and
\begin{equation}
\label{Eq2}
\gamma(s)=(C+iD)+(1-e^{-2\pi is})(A+iB)
\quad
\text{when}
\quad
L^2=-4\pi\DD.
\end{equation}
\end{theorem}
\begin{remark}
\label{dwa-trzy}
Let $A_j,B_j,C_j$ and $D_j$, $j=1,2,\ldots,n$ be the components of the points $A,B,C$ and $D$ respectively.
In terms of real components of $\gamma$, \eqref{Eq1} can be written as
\begin{equation}
\label{Eq3}
\begin{aligned}
x_j(s) = C_j+A_j (1- \cos(2 \pi s)) + B_j \sin(2 \pi s) \\
y_j(s) = D_j+B_j (1- \cos(2 \pi s)) - A_j \sin(2 \pi s)
\end{aligned}
\end{equation}
and \eqref{Eq2} as
\begin{equation}
\label{Eq4}
\begin{aligned}
x_j(s) = C_j+A_j (1- \cos(2 \pi s)) - B_j \sin(2 \pi s) \\
y_j(s) = D_j+B_j (1- \cos(2 \pi s)) + A_j \sin(2 \pi s).
\end{aligned}
\end{equation}
That is, the curves $\gamma_j=(x_j,y_j)$
are circles of radius $\sqrt{A_j^2+B_j^2}$ passing through $(C_j,D_j)$ at $s=0$. In the case of
\eqref{Eq1} they are all positively oriented and in the case of \eqref{Eq2} they are all negatively oriented. In either case, they are
parametrized with constant angular speed.
\end{remark}
\begin{remark}
If we have two different circles of the form \eqref{Eq1} having the same radius, then one can be mapped onto the other one by a composition of translations 
and a unitary map of $\bbbc^n$. See the proof of Proposition~\ref{Natalia}. The same comment applies to circles of the form \eqref{Eq2}.
\end{remark}
\begin{proof}
Let $\gamma=(x_1,\ldots,x_n,y_1,\ldots,y_n):[0,1]\to\bbbr^{2n}$ be a closed rectifiable curve. By translating the curve, we may assume without loss of generality that $\gamma(0)=0$. 
It suffices to prove \eqref{e2015} along with equations \eqref{Eq3} and \eqref{Eq4} (with $C=D=0$) which are, as was pointed out in Remark~\ref{dwa-trzy},
equivalent to \eqref{Eq1} and \eqref{Eq2}.
Since the curve has constant speed, its speed equals the length of the curve, so
$$
\sum_{j=1}^n \dot{x}_j(s)^2 + \dot{y}_j(s)^2 =  L^2.
$$
In particular the functions $x_j$ and $y_j$ are $L$-Lipschitz continuous and $x_j(0)=y_j(0)=x_j(1)=y_j(1)=0$. 
Hence the functions $x_j,y_j$ extend to $1$-periodic Lipschitz functions on $\bbbr$, and so we can 
use Fourier series to investigate them. 
We will follow notation used in \cite{dym}.
For  a $1$-periodic function $f$ let
$$
\hat{f}(k)=\int_0^1 f(x)e^{-2\pi ikx}\, dx,
\quad
k\in\bbbz
$$
be its $k$th Fourier coefficient. By Parseval's identity,
\begin{eqnarray}
L^2 
& = &
\sum_{j=1}^n \int_0^1 |\dot{x}_j(s)|^2 + |\dot{y}_j(s)|^2 \, ds 
= 
\sum_{j=1}^n \sum_{k \in \mathbb{Z}} |\hat{\dot{x}}_j(k)|^2 + |\hat{\dot{y}}_j(k)|^2 \nonumber \\
& = & 
\sum_{j=1}^n \sum_{k \in \mathbb{Z}} 4 \pi^2 k^2 \left(|\hat{x}_j(k)|^2 + |\hat{y}_j(k)|^2 \right)  
\label{length}
\end{eqnarray}
Note that
$$
\DD=\DD_1+\ldots+\DD_n = \frac{1}{2} \sum_{j=1}^n \int_0^1 \left( \dot{y}_j(s) x_j(s) - \dot{x}_j(s) y_j(s) \right) \, ds.
$$
Since $\dot{x}_j$ and $\dot{y}_j$ are real valued, we have $\dot{x}_j(s) = \overline{\dot{x}_j(s)}$ and $\dot{y}_j(s) = \overline{\dot{y}_j(s)}$.
Thus we may apply Parseval's theorem to this pair of inner products to find
\begin{eqnarray}
\DD
& = &
\frac{1}{2} \sum_{j=1}^n \left( \sum_{k \in \mathbb{Z}} \overline{\hat{\dot{y}}_j(k)} \hat{x}_j(k) - 
\sum_{k \in \mathbb{Z}} \overline{\hat{\dot{x}}_j(k)} \hat{y}_j(k)\right) \nonumber \\
& = & 
\frac{1}{2}\sum_{j=1}^n \sum_{k \in \mathbb{Z}} 2 \pi ki \left( \overline{\hat{x}_j(k)} \hat{y}_j(k) - \overline{\hat{y}_j(k)} \hat{x}_j(k) \right) \nonumber \\
& = &
\pi \sum_{j=1}^n \sum_{k \in \mathbb{Z}} k\cdot 2\, \text{Im}\Big(\overline{\hat{y}_j(k)}\hat{x}_j(k)\Big),
\label{area}
\end{eqnarray}
since $i(\bar{z}-z)=2\,\text{Im}\, z$.
Subtracting \eqref{area} from \eqref{length} gives
\begin{eqnarray}
\frac{L^2}{4 \pi^2} - \frac{\DD}{\pi} 
& =& 
\sum_{j=1}^n \Big[ \sum_{k \in \mathbb{Z}} k^2 \Big( |\hat{x}_j(k)|^2 + |\hat{y}_j(k)|^2 \Big)
-
k \cdot 2\, \text{Im}\Big(\overline{\hat{y}_j(k)}\hat{x}_j(k)\Big)\Big] \nonumber\\
& = & 
\sum_{j=1}^n \Big[ \sum_{k \in \mathbb{Z}} (k^2 - |k|) \Big( |\hat{x}_j(k)|^2 + |\hat{y}_j(k)|^2 \Big) \nonumber\\
& &   + 
|k| \left( |\hat{y}_j(k)|^2 - 2\, \text{sgn}(k) \text{Im}\, \Big( \overline{\hat{y}_j(k)} \hat{x}_j(k)\Big) + |\hat{x}_j(k)|^2 \right)\Big] \nonumber\\
& = &
\label{B52}
\sum_{j=1}^n \Big[ \sum_{k \in \mathbb{Z}} (k^2 - |k|) \Big( |\hat{x}_j(k)|^2 + |\hat{y}_j(k)|^2 \Big)
+
\sum_{k\in\mathbb{Z}} |k| \big| \hat{y}_j(k) + i \, \text{sgn}(k) \hat{x}_j(k) \big|^2 \Big].
\end{eqnarray}
The last equality follows from the identity $|a+ib|^2=|a|^2-2\,\text{Im}(\bar{a}b)+|b|^2$ which holds for all $a,b \in \mathbb{C}$.
Since every term in this last sum is non-negative, it follows that $\frac{L^2}{4 \pi^2} - \frac{\DD}{\pi} \geq 0$.
Thus, we have $L^2 \geq 4\pi \DD$. Reversing the orientation of the curve, i.e. applying the above argument to $\tilde{\gamma}(t)=\gamma(1-t)$ gives
$L^2\geq -4\pi \DD$, so \eqref{e2015} follows. 

Equality in \eqref{e2015} holds if and only if either $L^2=4\pi\DD$ or $L^2=-4\pi\DD$. 
We will first consider the case $L^2=4\pi\DD$.
This equality will occur if and only if each of the two sums contained inside the brackets in \eqref{B52} equals zero.
Since $k^2-|k|>0$ for $|k|\geq 2$,
the first of the two sums vanishes if and only if 
$\hat{x}_j(k)=\hat{y}_j(k)=0$ for every $|k|\geq 2$ and $j=1,2,\ldots,n$.
Hence nontrivial terms in the second sum correspond to $k=\pm 1$, and
thus this sum vanishes if and only if 
$\hat{y}_j(\pm 1) = -i \, \text{sgn} (\pm 1) \hat{x}_j(\pm 1)$. That is, for every $j=1,\dots, n$,
\begin{equation} 
\label{equal}
\begin{aligned}
\hat{y}_j(1) = -i \, \hat{x}_j(1) \quad \text{and} \quad \hat{y}_j(- 1) = i \, \hat{x}_j(- 1).
\end{aligned}
\end{equation}
Now since each $x_j$ and $y_j$ is Lipschitz, their Fourier series converge uniformly on $[0,1]$. Note that the only non-zero terms in the Fourier series
appear when $|k|\leq 1$. Thus $L^2=4\pi\DD$ if and only if \eqref{equal} is satisfied and 
for every $s\in [0,1]$ and $j=1, \dots, n$
\begin{equation} 
\label{zero}
\begin{aligned}
x_j(s) = \hat{x}_j(-1) e^{-2 \pi i s} + \hat{x}_j(0) + \hat{x}_j(1) e^{2 \pi i s} \\
y_j(s) = \hat{y}_j(-1) e^{-2 \pi i s} + \hat{y}_j(0) + \hat{y}_j(1) e^{2 \pi i s}.
\end{aligned}
\end{equation}
In particular, $0=x_j(0)=\hat{x}_j(-1)+\hat{x}_j(0)+\hat{x}_j(1)$ 
and hence $\hat{x}_j(0)=-\hat{x}_j(-1)-\hat{x}_j(1)$ 
for each $j = 1, \dots, n$. This together with Euler's formula gives 
\begin{align*}
x_j(s) &= \hat{x}_j(-1)\big(e^{-2\pi i s}-1\big)+\hat{x}_j(1)\big(e^{2\pi is}-1\big) \\
&= -(\hat{x}_j(-1) + \hat{x}_j(1))(1 - \cos(2 \pi s)) + (-i \hat{x}_j(-1) + i \hat{x}_j(1)) \sin(2 \pi s) \\
&= -(\hat{x}_j(-1) + \hat{x}_j(1))(1 - \cos(2 \pi s)) - (\hat{y}_j(-1) + \hat{y}_j(1)) \sin(2 \pi s).
\end{align*}
The last equality follows from (\ref{equal}). Similarly, we have
$$
y_j(s) = -(\hat{y}_j(-1) + \hat{y}_j(1))(1 - \cos(2 \pi s)) + (\hat{x}_j(-1) + \hat{x}_j(1)) \sin(2 \pi s).
$$
If we write $A_j = -(\hat{x}_j(-1) + \hat{x}_j(1))$ and $B_j = -(\hat{y}_j(-1) + \hat{y}_j(1))$, then we have 
\begin{equation}
\label{param+}
\begin{aligned}
x_j(s) = A_j (1- \cos(2 \pi s)) + B_j \sin(2 \pi s) \\
y_j(s) = B_j (1- \cos(2 \pi s)) - A_j \sin(2 \pi s).
\end{aligned}
\end{equation}
Note that it follows directly from the definition of Fourier coefficients that the numbers $A_j,B_j$ are real.

The case $L^2=-4\pi \DD$ is reduced to the above case by reversing the orientation of $\gamma$ as previously described. In that case the curves
$\gamma_j$ are given by
\begin{equation}
\label{param++}
\begin{aligned}
x_j(s) = A_j (1- \cos(2 \pi s)) - B_j \sin(2 \pi s) \\
y_j(s) = B_j (1- \cos(2 \pi s)) + A_j \sin(2 \pi s).
\end{aligned}
\end{equation}
We proved that if $L^2=4\pi\DD$, then $\gamma$ is of the form \eqref{Eq3} and if $L^2=-4\pi\DD$, then it is of the form \eqref{Eq4}. 
In the other direction, a straightforward calculation shows that any curve of the form \eqref{Eq3} satisfies $L^2=4\pi\DD$ and any
curve of the form \eqref{Eq4} satisfies $L^2=-4\pi\DD$. This completes the proof.
\end{proof}

\begin{proof}[Proof of Theorem~\ref{main}]
Suppose first that $\Gamma=(\gamma,t):[0,1] \to \Heis^n$ is any horizontal curve of constant speed connecting the origin to the point $(0,0,+T)$, $T>0$.
Recall from \eqref{speed} that
$$
\sum_{j=1}^n \dot{x}_j(s)^2 + \dot{y}_j(s)^2 = \ell_{cc}(\Gamma)^2 =: L^2.
$$
Thus $\gamma:[0,1]\to\bbbr^{2n}$ is a closed curve of length $L$ parametrized by arc-length. Moreover $\gamma(0)=0$.

If $\DD$ is defined as in Theorem~\ref{CorIso},
it follows from  \eqref{area1} that
$$
T = 2 \sum_{j=1}^n \int_0^1 \left( \dot{x}_j(s) y_j(s) - \dot{y}_j(s) x_j(s) \right) \, ds=-4\DD
$$
so $\DD<0$ and $L^2\geq \pi T$ by Theorem~\ref{CorIso}. Now $\Gamma$ is a geodesic if and only if 
$L^2=\pi T=-4\pi \DD$ which is the case of the equality in the isoperimetric inequality \eqref{e2015}. 
We proved above that this is equivalent to the components of $\gamma$ satisfying \eqref{param++},
and this is the $(0,0,+T)$ case of \eqref{param}. 
One may also easily check that $4\pi\sum_{j=1}^n (A_j^2+B_j^2)=-4 \DD = T$.

Suppose now that $\Gamma:[0,1]\to\Heis^n$ is any horizontal curve of constant speed connecting the origin to the point $(0,0,-T)$, $T>0$.
Then $\Gamma= (x(s),y(s),t(s))$ is a geodesic if and only if
$$
\tilde{\Gamma}(s)=(\tilde{x}(s),\tilde{y}(s),\tilde{t}(s)) = (x(1-s),y(1-s),t(1-s) +T)
$$
is a geodesic connecting $\tilde{\Gamma}(0) = (0,0,0)$ and $\tilde{\Gamma}(1) = (0,0,T)$ since reversing a curve's parametrization does not change its length
and since the mapping $(x,y,t)\mapsto (x,y,t+T)$ (the vertical lift by $T$) is an isometry on $\Heis^n$.
Therefore $\tilde{\Gamma} = (\tilde{x},\tilde{y},\tilde{t})$ must have the form \eqref{param++}. Hence the $(0,0,-T)$ case of \eqref{param}
follows from \eqref{param++} by replacing $s$ with $1-s$.

The formula for the
$t$ component of $\Gamma$ follows from \eqref{area1}; the integral is easy to compute due to numerous cancellations.
\end{proof}

Using the complex notation as in Theorem~\ref{CorIso}, the geodesics from Theorem~\ref{main} connecting the origin to $(0,0,\pm T)$, $T>0$ can be represented as
\begin{equation}
\label{complex}
\Gamma(s)=\Big(\big(1-e^{\mp 2\pi is}\big)(A+iB),t(s)\Big)
\end{equation}
where $A=(A_1,\ldots,A_n)$, $B=(B_1,\ldots,B_n)$ are such that
$4\pi |A+iB|^2=T$ and
$$
t(s)=\pm T\left(s-\frac{\sin(2\pi s)}{2\pi}\right).
$$
Theorem~\ref{main} and a discussion preceding it describes geodesics connecting the origin to points either 
on the $t$-axis $(0,0,\pm T)$, $T>0$ or
in $\bbbr^{2n}\times \{ 0\}$. The question now is how to describe geodesics connecting the origin to a point $q$ which is neither on the $t$-axis nor in $\bbbr^{2n}\times \{ 0\}$.
It turns out that geodesics described in Theorem~\ref{main} cover the entire space $\Heis^n\setminus (\bbbr^{2n}\times\{ 0\})$ and we have

\begin{corollary}
\label{cor}
For any $q\in\Heis^n$ which is neither in the $t$-axis nor in the subspace $\bbbr^{2n}\times\{ 0\}$ there is a unique geodesic connecting the origin to $q$. 
This geodesic is a part of a geodesic connecting the origin to a point on the $t$-axis.
\end{corollary}
\begin{proof}
Let $q=(c_1,\ldots,c_n,d_1,\ldots,d_n,h)$ be such that $h\neq 0$ and $c_j,d_j$ are not all zero. We can write $q=(c+id,h)\in\bbbc^n\times\bbbr$.
First we will construct a geodesic $\Gamma_q$ given by \eqref{complex} so that $\Gamma_q(s_0)=q$ for some $s_0\in (0,1)$. Clearly, the curve $\Gamma_q \big|_{[0,s_0]}$
will be part of a geodesic connecting the origin to a point on the $t$-axis. Then we will prove that this curve is a unique geodesic 
(up to a reparametrization) connecting the origin to $q$. Assume that $h>0$ (the case $h<0$ is similar). We will find a geodesic 
passing through $q$ that connects $(0,0,0)$ to $(0,0,T)$, for some $T>0$. (If $h<0$ we find $\Gamma$ that connects $(0,0,0)$ to $(0,0,-T)$.)
It suffices to show that there is a point $A+iB\in\bbbc^n$ such that the system of equations
\begin{equation}
\label{system}
\big(1-e^{-2\pi is}\big)(A+iB)=c+id,
\qquad
4\pi |A+iB|^2\left(s-\frac{\sin(2\pi s)}{2\pi}\right) = h
\end{equation}
has a solution $s_0\in (0,1)$. We have
$A+iB=(c+id)/(1-e^{-2\pi is})$ and hence
\begin{equation}
\label{solution}
2\pi\, \frac{|c+id|^2}{1-\cos(2\pi s)}\left(s-\frac{\sin(2\pi s)}{2\pi}\right) = h.
\end{equation}
This equation has a unique solution $s_0\in (0,1)$ because the function on the left hand side
is an increasing diffeomorphism of $(0,1)$ onto $(0,\infty)$. We proved that, among geodesics connecting $(0,0,0)$ to
points $(0,0,T)$, $T>0$, there is a unique geodesic $\Gamma_q$ passing through $q$.
Suppose now that $\tilde{\Gamma}$ is any geodesic connecting $(0,0,0)$ to $q$. Gluing $\tilde{\Gamma}$
with $\Gamma_q\big|_{[s_0,1]}$ we obtain a geodesic connecting $(0,0,0)$ to $(0,0,T)$ and hence (perhaps after a reparametrization)
it must coincide with $\Gamma_q$. This proves uniqueness of the geodesic $\Gamma_q \big|_{[0,s_0]}$.
\end{proof}

We will now use the proof of Corollary~\ref{cor} to find a formula for the Carnot-Carath\'eodory distance between $0$ and $q=(z,h)$, $z\neq 0$, $h>0$.
We will need this formula in the next section. Let
\begin{equation}
\label{H(s)}
H(s)=
\frac{2\pi}{1-\cos(2\pi s)}\left(s-\frac{\sin(2\pi s)}{2\pi}\right) : (0,1)\to (0,\infty)
\end{equation}
be the diffeomorphism of $(0,1)$ onto $(0,\infty)$ described in \eqref{solution}. Let
$$
\Gamma(s)=\left(\left(1-e^{-2\pi is}\right)(A+iB),t(s)\right)
$$
be the geodesic from the proof of Corollary~\ref{cor} that passes through $q$ at $s_0\in (0,1)$. We proved that $s_0$ is a solution to \eqref{solution}
and hence $s_0$ is a function of $q$ given by
$$
s_0(q)=H^{-1}(h|z|^{-2}).
$$
Note that $A+iB=z/(1-e^{-2\pi is_0})$, so
$$
|A+iB|=\frac{|z|}{\sqrt{2(1-\cos(2\pi s_0))}}\, .
$$
Hence
$$
\sqrt{\sum_{j=1}^n \dot{x}_j^2(s) +\dot{y}_j^2(s)} = L = \sqrt{\pi T} =
2\pi |A+iB| =\frac{2\pi |z|}{\sqrt{2(1-\cos(2\pi s_0))}}
$$
where $L$ is the length of $\Gamma$ and $\Gamma(1)=(0,0,T)$.
Therefore
\begin{equation}
\label{cc0q}
d_{cc}(0,q) = \int_0^{s_0} \sqrt{\sum_{j=1}^n \dot{x}_j^2(s) +\dot{y}_j^2(s)}\, ds =
\frac{2\pi s_0 |z|}{\sqrt{2(1-\cos(2\pi s_0))}}\, .
\end{equation}

\section{Analyticity of the Carnot-Carath\'eodory metric} 
\label{analytic}
The center of the Heisenberg group $\bbbh^n$ is
$Z = \{ (z,h) \in \bbbh^n \; | \; z = 0 \}$.
It is well known that the distance function in $\bbbh^n$ is $C^{\infty}$ smooth away from the center \cite{montithesis,monti}, but
through the use of \eqref{param}, we will now see that this distance function is actually real analytic. 
\begin{theorem}
\label{main2}
The Carnot-Carath\'eodory distance $d_{cc}:\bbbr^{2n+1}\times\bbbr^{2n+1}\to\bbbr$ is real analytic on the set
$$
\left\{(p,q)\in\bbbh^n\times\bbbh^n=\bbbr^{2n+1}\times\bbbr^{2n+1}:\, q^{-1}*p\not\in Z\right\}.
$$
\end{theorem}
\begin{proof}
In the proof we will make a frequent use of a well known fact that a composition of 
real analytic functions is analytic, \cite[Proposition~2.2.8]{krantzp}.
It suffices to prove that the function $d_0(p) = d_{cc}(0,p)$ is real analytic on $\bbbh^n \setminus Z$.
Indeed, $w(p,q) = q^{-1}*p$ is real analytic as it is a polynomial. Also, $d_{cc}(p,q) = (d_0 \circ w)(p,q)$, 
so real analyticity of $d_0$ on $\bbbh^n \setminus Z$ will imply that $d_{cc}$ is real analytic on 
$w^{-1}( \bbbh^n \setminus Z ) = \{ (p,q) \in \bbbh^n \times \bbbh^n \, | \, q^{-1} * p \notin Z \}$. 

Define $H:(-1,1) \to \bbbr$ as
\begin{equation} 
\label{H}
H(s) = \frac{2 \pi}{1 - \cos(2 \pi s)} \left( s-\frac{\sin(2 \pi s)}{2 \pi} \right)=
\frac{\frac{2\pi s}{3!}-\frac{(2\pi s)^3}{5!}+\frac{(2\pi s)^5}{7!}-\ldots}{\frac{1}{2!}-\frac{(2\pi s)^2}{4!}+\frac{(2\pi s)^4}{6!}-\ldots}\, .
\end{equation}
Here, we divided by a common factor of $(2 \pi s)^2$ in the two power series on the right hand side.
That is, the denominator equals $(1 - \cos(2 \pi s))(2 \pi s)^{-2}$ which does not vanish on $(-1,1)$.
This implies that $H$ is real analytic on $(-1,1)$.
Indeed, considering $s$ as a complex variable, we see that $H(s)$ is holomorphic (and hence analytic) in an open set containing 
$(-1,1)$ as a ratio of two holomorphic functions with non-vanishing denominator.

As we pointed out in \eqref{H(s)}, the function $H$ is an increasing diffeomorphism of $(0,1)$ onto $(0,\infty)$. Since it is odd and $H'(0)=2\pi/3\neq 0$,
$H$ is a real analytic diffeomorphism of $(-1,1)$ onto $\bbbr$.
Again, using a holomorphic function argument we see that $H^{-1}:\bbbr\to (-1,1)$
is a real analytic. 

The function $z\mapsto |z|^{-2}$ is analytic on $\bbbr^{2n}\setminus\{ 0\}$ (as a composition of a polynomial $z\mapsto |z|^2$ and an analytic function $1/x$), 
so the function $(z,h)\mapsto h|z|^{-2}$ is analytic in $\bbbh^n\setminus Z$. Hence also $s_0(q)=H^{-1}(h|z|^{-2})$ is analytic on $\bbbh^n\setminus Z$.

Fix $q=(z,h) \in \bbbh^n \setminus Z$ with $h > 0$. 
Then by \eqref{cc0q}
\begin{equation}
\label{444}
d_0(q) =  \frac{2\pi s_0 |z|}{\sqrt{2(1-\cos(2\pi s_0))}}.
\end{equation}
Since $H(s_0)=h|z|^{-2}$, formula \eqref{H} yields
$$
2\pi s_0=(1-\cos(2\pi s_0))h|z|^{-2}+\sin(2\pi s_0).
$$
Substituting $2\pi s_0$ in the numerator of the right hand side of \eqref{444} gives
$$
d_0(q) = \frac{h \sqrt{1 - \cos(2 \pi s_0)}}{\sqrt{2}|z|} + \frac{|z| \sin(2 \pi s_0)}{\sqrt{2} \sqrt{1 - \cos(2 \pi s_0)}} = h \sin(\pi s_0) |z|^{-1} + |z|\cos(\pi s_0)
$$
where we used the trigonometric identities
$$
\sqrt{1-\cos(2 \pi s_0)} = \sqrt{2} |\sin(\pi s_0)| = \sqrt{2} \sin(\pi s_0) \quad \text{and} \quad \frac{\sin(2 \pi s_0)}{\sin(\pi s_0)} = 2 \cos(\pi s_0).
$$
To treat the case $h\leq 0$
let us define $s_0(q)=H^{-1}(h|z|^{-2})$ for any $q=(z,h)$, $z\neq 0$. Previously we defined $s_0(q)$ only when $h>0$.
It is easy to check that the mapping 
$$
q=(x,y,t)=(z,t)\mapsto \bar{q}=(\bar{z},-t)=(x,-y,-t)
$$ 
is an isometry of the Heisenberg group, so $d_0(q)=d_0(\bar{q})$.

If $h<0$ and $\bar{q}=(\bar{z},-h)$, then 
$$
s_0(q)=H^{-1}(h|z|^{-2}) = -H^{-1}(-h|\bar{z}|^{-2}) = -s_0(\bar{q})
$$
and hence
$$
d_0(q)=d_0(\bar{q})=-h\sin(\pi s_0(\bar{q}))|\bar{z}|^{-1}+|\bar{z}|\cos (\pi s_0(\bar{q})) =
h\sin(\pi s_0(q))+|z|\cos(\pi s_0(q)).
$$

In the case $h=0$, $\Gamma$ is a straight line in $\bbbr^{2n}$ from the origin to $q$, and so $d_0(q) = |z|$.

Therefore 
$$
d_0(q) = h \sin(\pi s_0(q))|z|^{-1} + |z|\cos(\pi s_0(q))=h \sin(\pi H^{-1}(h|z|^{-2}))|z|^{-1} + |z|\cos(\pi H^{-1}(h|z|^{-2}))
$$ 
for every $q = (z,h) \in \bbbh^n \setminus Z$, and so $d_0$ is analytic on $\bbbh^n \setminus Z$.
\end{proof}
We also proved
\begin{corollary}
\label{2015}
For $z\neq 0$, the Carnot-Carath\'eodory distance between the origin $(0,0)$ and $(z,h)$, $z\neq 0$ equals
$$
d_{cc}((0,0),(z,h)) =
h \sin(\pi H^{-1}(h|z|^{-2}))|z|^{-1} + |z|\cos(\pi H^{-1}(h|z|^{-2})).
$$
\end{corollary}

\section{Classification of non-unique geodesics} 
\label{rotation}

Any point $(0,0,\pm T)$, $T>0$ on the $t$ axis can be connected to the origin by infinitely many geodesics. The purpose of this section is to show that 
all such geodesics are actually obtained from one geodesic by a linear mapping which fixes the $t$-axis. This map is an isometry of $\Heis^n$
and also an isometry of $\bbbr^{2n+1}$.

\begin{proposition}
\label{Natalia}
If $\Gamma_1:[0,1] \to \Heis^n$ and $\Gamma_2:[0,1] \to \Heis^n$ are constant-speed geodesics with $\Gamma_1(0)=\Gamma_2(0)=(0,0,0)$ and 
$\Gamma_1(1)=\Gamma_2(1)=(0,0,\pm T)$ with $T > 0$, then we can write $\Gamma_2 = V \circ \Gamma_1$ where $V$ is a isometry in $\Heis^n$ which fixes the $t$-coordinate.
The map $V$ is also an isometry of $\bbbr^{2n+1}$, specifically a rotation about the $t$-axis.
\end{proposition}
\begin{proof}
Consider geodesics $\Gamma_1 = (\gamma_1,t)$ and $\Gamma_2 = (\gamma_2,t)$ defined in the statement of the proposition.
As in the discussion before Corollary \ref{cor}, we consider $\gamma_1$ and $\gamma_2$ as functions into $\mathbb{C}^n$ rather than into $\bbbr^{2n}$ and write 
$$
\gamma_1(s) = \left( 1 - e^{\mp 2 \pi i s} \right)(A+iB),
\qquad
\gamma_2(s) = \left( 1 - e^{\mp 2 \pi i s} \right)(C+iD)
$$
where $4\pi|A+iB|^2=4\pi |C+iD|^2=T$.
We claim that there is a unitary matrix $U \in \text{U}(n,\mathbb{C})$ such that $U(A+iB)=C+iD$.
Indeed, for any $0\neq z \in \bbbc^n$, use the Gram-Schmidt process to extend $\{ z/|z| \}$ to an orthonormal basis of 
$\bbbc^n$ and define $W_z$ to be the matrix whose columns are these basis vectors. Here, we consider orthogonality with respect to the 
standard Hermitian inner product $ \langle u,v \rangle_{\bbbc} = \sum_{j=1}^n u_j \overline{v}_j$. 
Then $W_z \in \text{U}(n,\mathbb{C})$ and $W_z e_1 = z/|z|$ where $\{ e_1, \dots , e_n \}$ is the standard basis of $\bbbc^n$. 
Thus the desired operator is $U = W_{C+iD} \circ W_{A+iB}^{-1}$.

Define the linear map $V:\bbbc^n \times \bbbr \to \bbbc^n \times \bbbr$ by $V(z,t)=(Uz,t)$. Since
$$
U \left( \left( 1 - e^{\mp 2 \pi i s} \right)(A+iB) \right) = \left( 1 - e^{\mp 2 \pi i s} \right)(C+iD),
$$
for every $s \in [0,1]$ and since $V$ fixes the $t$-component of $\bbbc^n \times \bbbr$, we have $V \circ \Gamma_1 = \Gamma_2$.

We now prove that $V$ is an isometry on $\Heis^n$. Indeed, suppose $p,q \in \Heis^n$ and $\Gamma=(\gamma,t):[0,1] \to \Heis^n$ is a geodesic connecting them. 
Then $V \circ \Gamma = (U \circ \gamma, t)$. Since $\Gamma$ is horizontal, it is easy to check that
$\dot{t}(s) = 2 \text{ Im} \langle \gamma(s), \dot{\gamma}(s) \rangle_{\bbbc}$ for almost every $s \in [0,1]$. 
Unitary operators preserve the standard inner product on $\bbbc^n$, and so
\begin{align*}
\dot{t}(s) 
= 
2 \text{ Im} \left\langle \gamma(s), \dot{\gamma}(s) \right\rangle_{\bbbc} 
&= 2 \text{ Im} \left\langle (U \circ \gamma)(s), (U \circ \dot{\gamma})(s) \right\rangle_{\bbbc} \\
&= 2 \text{ Im} \left\langle (U \circ \gamma)(s), \frac{d}{ds}(U \circ \gamma)(s) \right\rangle_{\bbbc}
\end{align*}
for almost every $s \in [0,1]$. That is, $V \circ \Gamma$ is horizontal. Also, 
$$
\ell_H(\Gamma) = \int_0^1 \sqrt{\langle \dot{\gamma}(s),\dot{\gamma}(s) \rangle_{\bbbc}} \; ds
= 
\int_0^1 \sqrt{\langle (U \circ \dot{\gamma})(s),(U \circ \dot{\gamma})(s) \rangle_{\bbbc}} \; ds 
= 
\ell_H(V \circ \Gamma).
$$ 
Thus $d_{cc}(Vp,Vq) \leq \ell_H(\Gamma) = d_{cc}(p,q)$. Since $U$ is invertible and $U^{-1} \in \text{U}(n,\bbbc)$, we may argue similarly to show that 
$d_{cc}(p,q) = d_{cc}(V^{-1}Vp,V^{-1}Vq) \leq d_{cc}(Vp,Vq)$, and so $V$ is an isometry on $\Heis^n$.
Clearly unitary transformations of $\bbbc^n$ are also orientation preserving isometries of $\bbbr^{2n}$ and hence $V$ is a rotation of $\bbbr^{2n+1}$
about the $t$-axis.
\end{proof}

\section{Appendix}

\begin{proof}[Proof of Proposition~\ref{prop}]

(1) The components of any horizontal curve $\Gamma$ are absolutely continuous, and so their derivatives are integrable. 
Thus the inequality $\ell_{cc}(\Gamma) \leq \ell_H(\Gamma) < \infty$ 
yields rectifiability of horizontal curves.
It remains to prove that $\ell_{cc}(\Gamma) \geq \ell_H(\Gamma)$. 
Extend the Riemannian tensor defined on the horizontal distribution $H\Heis^n$ to a Riemannian tensor $g$ in $\bbbr^{2n+1}$. For example we may require
that the vector fields $X_j,Y_j,T$ are orthonormal at every point of $\bbbr^{2n+1}$.
The Riemannain tensor $g$ defines a metric $d_g$ in $\bbbr^{2n+1}$ in a standard way as the infimum of lengths of curves connecting 
two given points, where the length of an absolutely continuous curve $\alpha:[a,b] \to \bbbr^{2n+1}$ is defined as the integral
$$
\ell_g(\alpha)= 
\int_a^b \sqrt{g \left( \dot{\alpha}(s),\dot{\alpha}(s) \right)} \, ds.
$$
This is the same approach that was used to define the Carnot-Carath\'eodory metric. For horizontal 
curves $\Gamma$, we have $\ell_g(\Gamma)=\ell_H(\Gamma)$, and so it is obvious that 
$d_g(p,q)\leq d_{cc}(p,q)$ since we now take an infimum over a larger class of curves.
It is a well known fact in Riemannian geometry that for an absolutely continuous curve $\alpha:[a,b] \to \bbbr^{2n+1}$
$$
\ell_g(\alpha) = \sup\sum_{i=0}^{n-1} d_g(\alpha(s_i),\alpha(s_{i+1})),
$$
where the supremum is taken over all $n\in\bbbn$ and all partitions $a=s_0\leq s_1\leq\ldots\leq s_n=b$ as before.
Hence if $\Gamma:[a,b]\to\bbbr^{2n+1}$ is horizontal we have
$$
\ell_H(\Gamma)=\ell_g(\Gamma)=\sup\sum_{i=0}^{n-1} d_g(\Gamma(s_i),\Gamma(s_{i+1}))\leq \sup\sum_{i=0}^{n-1} d_{cc}(\Gamma(s_i),\Gamma(s_{i+1}))=\ell_{cc}(\Gamma).
$$
(2)
If $\Gamma:[a,b]\to\Heis^n$ is Lipschitz, then by \eqref{EH} it is also
Lipschitz with respect to the Euclidean metric in $\bbbr^{2n+1}$. 
Thus it is absolutely continuous and differentiable a.e. It remains to show that
$\dot{\Gamma}(s)\in H_{\Gamma(s)}\Heis^n$ a.e. We will actually show that this is true whenever $\Gamma$ is differentiable at $s$. 

Suppose $\Gamma$ is differentiable at a point $s \in (a,b)$ and let $i_0\in\mathbb{N}$ be such that $s+2^{-i_0}\leq b$. Then
$d_{cc}(\Gamma(s + 2^{-i}) , \Gamma(s)) \leq L 2^{-i}$ for some $L > 0$ and all $i\geq i_0$.
By the definition of the Carnot-Carath\'eodory metric on $\Heis^n$, there is some horizontal curve
$\eta^i:[0,2^{-i}] \to \Heis^n$ which connects $\Gamma(s)$ to $\Gamma(s + 2^{-i})$ whose length approximates the distance between them. 
That is, we can choose a horizontal curve $\eta^i=(x^i,y^i,t^i)$ so that $\eta^i(0) = \Gamma(s)$,
$\eta^i(2^{-i}) = \Gamma(s+2^{-i})$ and 
$\ell_H(\eta^i)<2L 2^{-i}$.
After a reparameterization, we may assume that $\eta^i$ has constant speed $|\dot{\eta}^i|_H< 2L$ on $[0,2^{-i}]$.
Since $\eta^i$ is horizontal, we can write
$$
\dot{\eta}^i(\tau) = \sum_{j=1}^n \dot{x}_j^i(\tau) X_j(\eta^i(\tau)) + \dot{y}_j^i(\tau) Y_j(\eta^i(\tau))
$$
for almost every $\tau \in [0,2^{-i}]$.
Now $\Gamma(s+2^{-i})-\Gamma(s) = \eta^i(2^{-i})-\eta^i(0) = \int_0^{2^{-i}} \dot{\eta}^i(\tau) \, d\tau$, and so
\begin{eqnarray}
\frac{\Gamma(s+2^{-i})-\Gamma(s)}{2^{-i}} 
& = &
\frac{1}{2^{-i}} \int_0^{2^{-i}} \sum_{j=1}^n \big(\dot{x}_j^i(\tau) X_j(\eta^i(\tau)) + \dot{y}_j^i(\tau) Y_j(\eta^i(\tau))\big) \, d\tau \nonumber \\
& = &
\frac{1}{2^{-i}} \int_0^{2^{-i}} \sum_{j=1}^n \dot{x}_j^i(\tau) \big(X_j(\eta^i(\tau))-X_j(\Gamma(s))\big) \, d\tau \label{one}\\
& & + \frac{1}{2^{-i}} \int_0^{2^{-i}} \sum_{j=1}^n \dot{y}_j^i(\tau) \big(Y_j(\eta^i(\tau))-Y_j(\Gamma(s)))\big) \, d\tau \label{two} \\
& & + \frac{1}{2^{-i}} \int_0^{2^{-i}} \sum_{j=1}^n \big(\dot{x}_j^i(\tau) X_j(\Gamma(s)) + \dot{y}_j^i(\tau) Y_j(\Gamma(s))\big) \, d\tau \nonumber
\end{eqnarray}
The images of all the curves $\eta^i$ are contained in a compact subset of $\bbbr^{2n+1}$. Hence \eqref{loclip} yields
$$
\sup_{\tau\in [0,2^{-i}]} |\eta^i(\tau)-\Gamma(s)|\leq \ell_E(\eta^i)
\leq C(K)\ell_H(\eta^i)\leq C(K)L2^{-i+1}\to 0
\quad
\text{as $i\to\infty$.}
$$
Here $\ell_E(\eta^i)$ stands for the Euclidean length of $\eta^i$. Observe also that the functions $\dot{x}^i_j$ and $\dot{y}^i_j$
are bounded almost everywhere by the speed $|\dot{\eta}^i|_H$ which is less than $2L$.
This and the continuity of the vector fields $X_j$, $Y_j$ immediately imply that
the sums \eqref{one} and \eqref{two} converge to zero as $i\to\infty$. Also, by again using the uniform boundedness of the functions 
$\dot{x}_j^i$ and $\dot{y}_j^i$, we conclude that for some subsequence and all $j=1,2,\ldots,n$ the averages
$2^i \int_0^{2^{-i}} \dot{x}_j^i$ and $2^i \int_0^{2^{-i}} \dot{y}_j^i$ converge to some constants $a_j$ and $b_j$ respectively.
Denote such a subsequence by $i_k$.
Therefore, since $\dot{\Gamma}(s)$ exists we have
$$
\dot{\Gamma}(s) = \lim_{k \to \infty} \frac{\Gamma(s+2^{-i_k})-\Gamma(s)}{2^{-i_k}} = 
\sum_{j=1}^n (a_j X_j(\Gamma(s)) + b_j Y_j(\Gamma(s))) \in H_{\Gamma(s)} \Heis^n.
$$
This proves (2).
Finally (3) follows from a general fact that a rectifiable curve in any metric space admits an arc-length parametrization with respect to which the curve is $1$-Lipschitz, see e.g.
\cite[Proposition~2.5.9]{burago}, \cite[Theorem~3.2]{hajlasz}.
\end{proof}

\end{document}